\documentclass[12pt]{article}
\usepackage{graphicx}
\usepackage[ruled,vlined]{algorithm2e}
\usepackage{appendix}
\usepackage{amsmath}
\usepackage{amssymb}
\usepackage{amsthm}
\usepackage{multirow}
\usepackage{color}
\usepackage{longtable}
\usepackage{array}
\usepackage{url}
\usepackage{booktabs}
\usepackage{float}
\usepackage{mathtools}
\usepackage{tikz}
\usepackage{caption}
\usepackage{diagbox}
\usepackage{enumerate}
\usepackage{mathrsfs}
\usepackage[bookmarks=true]{hyperref}

\newtheorem{theorem}{Theorem}[section]

\newtheorem{lemma}[theorem]{Lemma}

\newtheorem{problem}[theorem]{Problem}
\newtheorem{conjecture}[theorem]{Conjecture}

\theoremstyle{definition}

\title{On color isomorphic subdivisions}
\author{Zixiang Xu$^{\text{a}}$\thanks{e-mail: zxxu8023@qq.com.} and Gennian Ge$^{\text{a,}}$\thanks{e-mail: gnge@zju.edu.cn. Research supported by the National Natural Science Foundation of China under Grant No. 11971325, National Key Research and Development Program of China under Grant  No. 2018YFA0704703,  and Beijing Scholars Program.}\\
\footnotesize $^{\text{a}}$ School of Mathematical Sciences, Capital Normal University, Beijing 100048, China.\\
}

\begin{document}

\date{}

\maketitle

\begin{abstract}
  Given a graph $H$ and an integer $k\geqslant 2$, let $f_{k}(n,H)$ be the smallest number of colors $C$ such that there exists a proper edge-coloring of the complete graph $K_{n}$ with $C$ colors containing no $k$ vertex-disjoint color isomorphic copies
of $H$. In this paper, we prove that $f_{2}(n,H_{t})=\Omega(n^{1+\frac{1}{2t-3}})$ where $H_{t}$ is the $1$-subdivision of the complete graph $K_{t}$. This answers a question of Conlon and Tyomkyn (arXiv: 2002.00921).
\end{abstract}

\medskip


\smallskip


\section{Introduction}\label{sec:Introduction}
 Recently, Conlon and Tyomkyn~\cite{Conlon2020} initiated the study of a new problem on extremal graph theory, which aims to find two or more vertex-disjoint color isomorphic copies of some given graph in proper edge-colorings of complete graphs. Formally, we say that two vertex-disjoint copies of a graph $H$ in a coloring of $K_{n}$ are color isomorphic if there exists an isomorphism between them preserving the colors. For an integer $k\geqslant 2$ and a graph $H$, let $f_{k}(n,H)$ be the smallest number of colors $C$ such that there exists a proper edge-coloring of the complete graph $K_{n}$ with $C$ colors containing no $k$ vertex-disjoint color isomorphic copies of $H$. Obviously we have $n-1\leqslant f_{k}(n,H)\leqslant\binom{n}{2}$ since the coloring of $K_{n}$ is proper. Conlon and Tyomkyn~\cite{Conlon2020} first verified that finding $f_{k}(n,H)$ is indeed an extremal problem. Hence one may ask the following question.
  \begin{problem}\label{problem:MAIN}
Given a graph $H$ and an integer $k\geqslant 2$, determine the order of growth of $f_{k}(n,H)$ as $n\rightarrow\infty.$
\end{problem}
   In \cite{Conlon2020}, Conlon and Tyomkyn showed various general results about the function $f_{k}(n,H)$, such as the following upper bounds.
\begin{theorem}[\cite{Conlon2020}]\label{thm:SomeknownResults} The followings hold.\medskip

 \noindent\emph{(i)} For any graph $H$ with $v$ vertices and $e$ edges,
  \begin{equation*}
    f_{k}(n,H)=O(\max\{n,n^{\frac{kv-2}{(k-1)e}}\}).
    \end{equation*}

    \noindent\emph{(ii)} For every graph $H$ containing a cycle, there exists $k=k(H)$ such that
\begin{equation*}
  f_{k}(n,H)=\Theta(n).
\end{equation*}
\end{theorem}
There were also some known results on this function. For example, Theorem~\ref{thm:SomeknownResults} (ii) is from the random algebraic method of Bukh~\cite{Bukh2015}. When $H$ is an even cycle, the above constant $k=k(H)$ obtained by the random algebraic method is likely very large due to the Lang-Weil bound~\cite{LangWeil1954}. Recently, Ge, Jing, Xu and Zhang~\cite{XuZhangJingGe2020} improved the constant $k=k(C_{4})$ by showing that $f_{3}(n,C_{4})=\Theta(n)$ via an algebraic construction. On the other hand, Conlon and Tyomkyn~\cite{Conlon2020} proved that $f_{2}(n,C_{6})=\Omega(n^{\frac{4}{3}}).$ Very recently, Janzer~\cite{Janzer2020} developed a new method for finding suitable cycles of given length and then obtained a general lower bound as follows.
\begin{theorem}[\cite{Janzer2020}]
  Let $k,\ell$ be fixed integers. Then we have
  \begin{equation*}
    f_{k}(n,C_{2\ell})=\Omega(n^{\frac{k}{k-1}\cdot\frac{\ell-1}{\ell}}).
  \end{equation*}
\end{theorem}
As a corollary of Theorem~\ref{thm:SomeknownResults} (i), one can see that if $e(H)\geqslant 2v(H)-2,$ then $f_{2}(n,H)=\Theta(n).$ Conlon and Tyomkyn~\cite{Conlon2020} asked how sharp this bound is and they also proved that $f_{2}(n,\theta_{3,\ell})=\Omega(n^{\frac{4}{3}}).$ Since $e(\theta_{3,\ell})=\frac{3}{2}v(\theta_{3,\ell})-3$, the above corollary cannot be improved, to say that, $e(H)\geqslant \frac{3}{2}v(H)-3$ implies that $f_{2}(n,H)=\Theta(n).$ They also suggested that an interesting test case for deciding whether this lower bound can be pushed closer to $2v(H)$ might be to study $f_{2}(n,H_{t})$, where $H_{t}$ is the $1$-subdivision of the complete
graph $K_{t}$. The main result of this paper answers their question as follows.

\begin{theorem}\label{thm:subdivision}
 Let $t\geqslant 3$ be a fixed integer. Then we have
 \begin{equation*}
   f_{2}(n,H_{t})=\Omega(n^{1+\frac{1}{2t-3}}).
 \end{equation*}
\end{theorem}
The proof of our main result is mainly based on the ideas in~\cite{JanzerEJC2019}. We modify them at some points and add some new ideas such as the deletion method.
Theorem~\ref{thm:subdivision} indicates that $f_{2}(n,H_{t})$ grows superlinearly with $n.$ Since $e(H_{t})=2v(H_{t})-2t,$ our result tells that, $e(H)\geqslant 2v(H)-2t$ with $t\geqslant 3$ does not imply that $f_{2}(n,H)=\Theta(n).$

\textbf{Notation}: The notations $o$, $O$, $\Omega,$ $\Theta$ have their usual asymptotic meanings. For a graph $G$ and subset $X\subseteq V(G)$, we denote $G[X]$ as the induced subgraph of $G$. Usually we denote $N_{G}(v)$ as the set of neighbors of $v$ in $G$ and denote $\text{deg}(v):=|N_{G}(v)|$ as the degree of $v$ in $G$.

\section{Proof of Theorem~\ref{thm:subdivision}}\label{sec:proof}

Suppose that $C=\gamma n^{1+\frac{1}{2t-3}}$, where $\gamma$ is a sufficiently small constant. Suppose also that $n$ is taken sufficiently
large. For convenience in our proof, we will also assume that $n$ is divided by $4$. First we need the following lemma, which will help us construct the auxiliary graph.

\begin{lemma}\label{lem:partition}
  Given a proper $C$-coloring $\chi$ of $G=K_{n},$ take a random equipartition of $V(G)$ into four parts $X_{1}$, $X_{2}$, $X_{3}$ and $X_{4}.$ Then the expected number of monochromatic matchings of the form $\{x_{1}x_{3},x_{2}x_{4}\}$ in $G$ is $\frac{1}{256}\sum_{c\in \chi}\binom{e_{c}}{2},$ where $e_{c}$ is the number of edges with color $c$ in $G$ and $x_{i}\in X_{i}$ for $i=1,2,3,4.$
\end{lemma}
 \begin{proof}[Proof of Lemma~\ref{lem:partition}]
   Given a proper $C$-coloring of $G=K_{n},$ take a random equipartition of $V(G)$ into four parts $X_{1}$, $X_{2}$, $X_{3}$ and $X_{4}.$ Since a monochromatic matching of size two in $G$ contains $4$ vertices, and the probability of each of such vertex in $X_{i}$ is $\frac{1}{4}.$ Hence, the expected number of such monochromatic matchings of the form $\{x_{1}x_{3},x_{2}x_{4}\}$ is $\frac{1}{4^{4}},$ where $x_{i}\in X_{i}$ for $i=1,2,3,4.$ Then by the linearity of expectation, the result is proved.
 \end{proof}

Next we construct the auxiliary graph $\mathcal{G}$ as follows. Given a proper $C$-coloring $\chi$ of $G=K_{n},$ we choose an equipartition of $V(G)$ into four parts $X_{1}$, $X_{2}$, $X_{3}$ and $X_{4},$ such that the number of monochromatic matchings of the form $\{x_{1}x_{3},x_{2}x_{4}\}$ in $G$ is at least $\frac{1}{256}\sum_{c\in \chi}\binom{e_{c}}{2},$ where $e_{c}$ is the number of edges with color $c$ in $G$ and $x_{i}\in X_{i}$ for $i=1,2,3,4.$ Let $\mathcal{G}$ be a bipartite graph, the vertex set $V(\mathcal{G})=(X_{1}\times X_{2})\cup(X_{3}\times X_{4})$ and $(x_{1},x_{2})\in X_{1}\times X_{2}$ is adjacent to $(x_{3},x_{4})\in X_{3}\times X_{4}$ if and only if $\{x_{1}x_{3},x_{2}x_{4}\}$ is a monochromatic matching in $G.$
 It is not hard to show that $|V(\mathcal{G})|=\frac{n^2}{8}$ and the number of edges in the auxiliary graph $\mathcal{G}$ is equal to the number of monochromatic matchings of the form $\{x_{1}x_{3},x_{2}x_{4}\}$ in $K_{n}.$ Hence, we have
\begin{equation*}
  |E(\mathcal{G})|\geqslant \frac{1}{256}\sum\limits_{c\in \chi}\binom{e_{c}}{2}\geqslant \frac{n^{4}}{1024C}>\frac{|V(\mathcal{G})|^{\frac{3}{2}-\frac{1}{4t-6}}}{1024\gamma},
\end{equation*}
where the second inequality uses the convexity of $\binom{x}{2}$ and the formula $\sum\limits_{c\in \chi}e_{c}=\binom{n}{2}.$

For the rest of the proof, we will show that $\mathcal{G}$ contains a copy of $H_{t}$ with the property that the vertices are pairwise disjoint sets. The next observation is inspired by Lemma~4.3 in~\cite{Janzer2020}, which is useful for making sure that the vertices are disjoint sets.
\begin{lemma}\label{lem:disjoint}
  For any vertex $S$ in $\mathcal{G}$ and any vertex $v$ in $G$, there is at most one vertex $T$ in $\mathcal{G}$ such that $ST$
is an edge in $\mathcal{G}$ and $v$ is in $T$.
\end{lemma}
\begin{proof}[Proof of Lemma~\ref{lem:disjoint}]
 Without loss of generality, assume that $S=(s_{1},s_{2})\in X_{1}\times X_{2}$ and $v\in X_{3}.$ If there are two distinct vertices $T_{1}=(v,t_{1})$ and $T_{2}=(v,t_{2})$ such that both of $T_{1}$ and $T_{2}$ are adjacent to $S,$ then the colors of edges $s_{2}t_{1}$ and $s_{2}t_{2}$ are the same, a contradiction.
\end{proof}

We say a graph $F$ is $K$-almost-regular if $\min\limits_{v\in V(F)}\text{deg}(v)\leqslant K\cdot\max\limits_{v\in V(F)}\text{deg}(v).$ Moreover, we say $F$ is a bipartite balanced graph with $V(F)=A\cup B$ if $\frac{1}{2}|B|\leqslant |A|\leqslant 2|B|.$  We shall use the following lemma, which has been used in many problems~\cite{ConlonJanzerLee2019, ConlonLeeIMRN2018, JanzerEJC2019, Jiang2012, SudakovTomon2019}.
\begin{lemma}\label{lem:balanced}
  For any positive constant $\alpha<1$, there exists $n_{0}$ such that if $n>n_{0}$, $c\geqslant 1$
and $F$ is an $n$-vertex graph with at least $cn^{1+\alpha}$ edges, then $F$ has a $K$-almost-regular
balanced bipartite subgraph $F'$ with $m$ vertices such that $m\geqslant n^{\frac{\alpha(1-\alpha)}{2(1+\alpha)}}$, $|E(F')|\geqslant \frac{c}{10}m^{1+\alpha}$ and $K=60\cdot 2^{1+\frac{1}{\alpha^{2}}}.$
\end{lemma}

 Let $\mathcal{G}$ be the auxiliary graph defined as above. By Lemma~\ref{lem:balanced}, we can find a $K$-almost-regular balanced bipartite subgraph $\mathcal{G}_{0}$ with $|V(\mathcal{G}_{0})|=n_{1}\geqslant |V(\mathcal{G})|^{\frac{\alpha(1-\alpha)}{2(1+\alpha)}},$ $|E(\mathcal{G}_{0})|\geqslant \frac{c_{0}}{10}n_{1}^{1+\alpha},$ where $\alpha=\frac{t-2}{2t-3},$ $K=60\cdot 2^{1+\frac{1}{\alpha^{2}}}$ and $c_{0}=\frac{1}{1024\gamma}\geqslant 1.$ Since the constant $\gamma$ is chosen to be sufficiently small, $c_{1}=\frac{c_{0}}{10}$ is a sufficiently large constant. To prove our main result, it suffices to show that in $\mathcal{G}_{0}$, there exists a copy of $H_{t}$ in which the vertices are pairwise disjoint, because if we can find a copy of $H_{t}$ in $\mathcal{G}_{0}$ such that the vertices are pairwise disjoint, then we can find two vertex-disjoint color isomorphic copies of $H_{t}$ in $G.$

 \begin{theorem}\label{thm:main}
 Let $\mathcal{G}_{0}$ be the subgraph of $\mathcal{G}$ defined as above. $\mathcal{G}_{0}$ contains a copy of $H_{t}$ in which the vertices are pairwise disjoint.
 \end{theorem}
Before we prove the above theorem, we collect a few results that will be useful to us. By the definition of $K$-almost-regular balanced bipartite graph, let $\mathcal{G}_{0}=A\cup B$ with $|B|=m$ and the degree of every vertex of $\mathcal{G}_{0}$ be between $\delta$ and $K\delta,$ where $\delta\geqslant c_{1}m^{\frac{t-2}{2t-3}}$ for some sufficiently large constant $c_{1}.$ Then we define the neighborhood graph $W_{A}$ on vertex set $A,$ where the weight of the pair $uv$ in $W_{A}$ is $d_{\mathcal{G}_{0}}(u,v)=|N_{\mathcal{G}_{0}}(u)\cap N_{\mathcal{G}_{0}}(v)|.$ Sometimes we also write the weight of the pair $uv$ as $W(u,v).$ Moreover, for a subset $U$ of $A,$ write $W(U)=\sum\limits_{uv\in\binom{U}{2}}d_{\mathcal{G}_{0}}(u,v).$

The following lemma on weighted graph $W_{A}$ of $\mathcal{G}_{0}$ has been shown in~\cite{ConlonLeeIMRN2018}.

\begin{lemma}[\cite{ConlonLeeIMRN2018}]\label{lem:ConlonLee}
Let $\mathcal{G}_{0}= A\cup B,$ be the bipartite graph with $|B|=m$ and minimum degree being at least $\delta$ in $A.$ Then for any subset $U\in A$ with $\delta|U|\geqslant 2m,$ we have
\begin{equation*}
  W(U)=\sum\limits_{uv\in\binom{U}{2}}d_{\mathcal{G}_{0}}(u,v)\geqslant \frac{\delta^{2}}{2m}\binom{|U|}{2}.
\end{equation*}
\end{lemma}

We further consider the weighted graph $W_{A}$ of $\mathcal{G}_{0}.$ For distinct vertices $u,v\in A$, we say that the edge $uv$ is light if $1\leqslant W(u,v)< 2\binom{t}{2},$ and that is heavy if $W(u,v)\geqslant 2\binom{t}{2}.$ Observe that if there is a copy of $K_{t}$ in $W_{A}$ formed by heavy edges, then there is a copy of $H_{t}$ in $\mathcal{G}_{0},$ in which the vertices of $H_{t}$ are pairwise disjoint. Based on the above observation, we can obtain the following lemma.

\begin{lemma}\label{lem:light}
  If $\mathcal{G}_{0}$ does not contain a copy of $H_{t}$ in which the vertices are pairwise disjoint, then for any subset $U\subseteq A$ with $|U|\geqslant \frac{8tm}{\delta}$ and $|U|\geqslant 2,$ the number of light edges in $W_{A}[U]$ is at least $\frac{\delta^{2}}{16t^{3}m}\binom{|U|}{2}.$
\end{lemma}
\begin{proof}[Proof of Lemma~\ref{lem:light}]
  By Lemma~\ref{lem:ConlonLee}, for any subset $U\subseteq A$ with $|U|\geqslant \frac{8tm}{\delta}$, we have
  \begin{equation*}
    W(U)\geqslant\frac{\delta^{2}}{2m}\binom{|U|}{2}\geqslant\frac{\delta^{2}}{8m}|U|^{2}\geqslant 8t^{2}m.
  \end{equation*}
   Let $B=\{b_{1},b_{2},\ldots,b_{m}\}$ and $h_{i}:=|N_{\mathcal{G}^{'}}(b_{i})|.$ Let $\mathcal{G}^{'}$ be the induced subgraph $\mathcal{G}_{0}[U,B]$ of $\mathcal{G}_{0}.$ Now suppose that for some $i,$ $h_{i}\geqslant 2(t-1).$ Since $\mathcal{G}_{0}$ does not contain a copy of $H_{t}$ in which the vertices are pairwise disjoint, there is no $K_{t}$ in the weighted graph $W_{A}[N_{\mathcal{G}^{'}}(b_{i})]$ formed by heavy edges. Hence by Tur\'{a}n theorem of $K_{t}$-free graph, the number of light edges in $W_{A}[N_{\mathcal{G}^{'}}(b_{i})]$ is at least
  \begin{equation*}
    \frac{1}{t-1}\binom{h_{i}}{2}\geqslant \frac{h_{i}^{2}}{4(t-1)}.
  \end{equation*}
  Moreover, note that
  \begin{equation*}
    \sum\limits_{i:h_{i}< 2(t-1)}\binom{h_{i}}{2}<4t^{2}m\leqslant \frac{W(U)}{2},
  \end{equation*}
  which implies that
  \begin{equation*}
    \sum\limits_{i:h_{i}\geqslant 2(t-1)}\binom{h_{i}}{2}\geqslant \frac{W(U)}{2}.
  \end{equation*}
  By the definition of light edge, every light edge is presented in at most $2\binom{t}{2}$ of the set $N_{\mathcal{G}^{'}}(b_{i}).$ Thus, the total number of edges in $W_{A}[U]$ is at least
  \begin{equation*}
    \frac{1}{2\binom{t}{2}}\sum\limits_{i:h_{i}\geqslant 2(t-1)}\frac{h_{i}^{2}}{4(t-1)}\geqslant\frac{W(U)}{8t^{3}}\geqslant \frac{\delta^{2}}{16t^{3}m}\binom{|U|}{2}.
  \end{equation*}
  The proof is finished.
\end{proof}

With the above tools in hand, now we are ready to prove Theorem~\ref{thm:subdivision}. Actually, it suffices to prove Theorem~\ref{thm:main}. In order to avoid ambiguity, we need to clarify the specific meaning of some expressions. When we say $N_{\mathcal{G}_{0}}(u)\cap N_{\mathcal{G}_{0}}(v)=\emptyset$ in $\mathcal{G}_{0}$, we mean that there is no pair of vertices $(x,y)\subseteq G$ as a vertex $S\in \mathcal{G}_{0}$ such that $S\in N_{\mathcal{G}_{0}}(u)\cap N_{\mathcal{G}_{0}}(v)$ in graph $\mathcal{G}_{0}.$ Moreover, we say two vertices $S,T\in V(\mathcal{G}_{0})$ do not share a vertex in $G,$ we mean that if $S=(s_{1},s_{2})\subseteq V(G)$ and $T=(t_{1},t_{2})\subseteq V(G),$ then $s_{i}\neq t_{j}$ with $i,j\in\{1,2\}.$

\begin{proof}[Proof of Theorem~\ref{thm:main}]
 If we can find a copy of $H_t$ in $\mathcal{G}_{0}$ with vertices $u_{1},u_{2},\ldots,u_{t}$ on one side and vertex
$v_{i,j}$ joined to $u_{i}$ and $u_j$ for each $1\leqslant i<j\leqslant t$. By Lemma~\ref{lem:disjoint}, for any $1\leqslant i<j\leqslant t$, $u_{i}$ and $u_{j}$ cannot share a vertex in $G$, since they have the common neighbor $v_{i,j}$. Similarly, $v_{i,j}$ and $v_{i,k}$ cannot share a vertex in $G$. Hence we only need to show that for any distinct $i,j,k,\ell,$ $v_{i,j}$ and $v_{k,\ell}$ cannot share a vertex in $G$.

We shall define $u_{1},u_{2},\ldots,u_{t-1}\in A$ recursively with the following properties.
\begin{enumerate}[(i)]
  \item For any distinct $i,j\leqslant t-1,$ $u_{i}$ and $u_{j}$ form a light edge in $W_{A}$.
  \item For any distinct $i,j,k\leqslant t-1,$ $N_{\mathcal{G}_{0}}(u_{i})\cap N_{\mathcal{G}_{0}}(u_{j})\cap N_{\mathcal{G}_{0}}(u_{k})=\emptyset$ in $\mathcal{G}_{0}.$
  \item For any distinct $i,j,k,\ell\leqslant t-1,$ the vertices in $N_{\mathcal{G}_{0}}(u_{i})\cap N_{\mathcal{G}_{0}}(u_{j})$ and the vertices in $N_{\mathcal{G}_{0}}(u_{k})\cap N_{\mathcal{G}_{0}}(u_{\ell})$ do not share an element in $G$.
  \item For each $1\leqslant i\leqslant t-1,$ the number of $u\in A$ with the property that for every $j\leqslant i,$ $u_{j}u$ is light is at least $(\frac{\delta^{2}}{64t^{3}m})^{i}\cdot |A|.$
\end{enumerate}

As we have discussed above, combining the properties (i), (ii) and (iv) together helps us to find a copy of $H_{t}$ in $\mathcal{G}_{0}$ and the property (iii) helps us to show the vertex-disjoint property of such $H_{t}$ in $\mathcal{G}_{0}.$

Since $\mathcal{G}_{0}$ is balanced, $|A|\geqslant \frac{m}{2}\geqslant \frac{8tm}{\delta}$ as $m$ is sufficiently large. By Lemma~\ref{lem:light}, there are at least $\frac{\delta^{2}}{16t^{3}m}\binom{|A|}{2}$ light edges in $A.$ So by pigeonhole principle, we can pick some vertex $u_{1}\in A$ such that the number of light edges $u_{1}u$ is at least $\frac{\delta^{2}}{16t^{3}m}(|A|-1)\geqslant \frac{\delta^{2}}{64t^{3}m}|A|.$ Suppose that for $2\leqslant \ell\leqslant t-1,$ $u_{1},u_{2},\ldots,u_{\ell-1}$ have been chosen with properties (i), (ii), (iii) and (iv). Let $U_{0}$ be the set of vertices $u\in A$ such that $u_{j}u$ is a light edge for every $j\leqslant \ell-1.$ By the property (iv), we have that $|U_{0}|\geqslant (\frac{\delta^{2}}{64t^{3}m})^{\ell-1}|A|.$ Let $U$ be consisted of those $u\in U_{0}$ with the following properties.
\begin{itemize}
  \item For all $1<i<j\leqslant\ell-1,$ $N_{\mathcal{G}_{0}}(u_{i})\cap N_{\mathcal{G}_{0}}(u_{j})\cap N_{\mathcal{G}_{0}}(u)=\emptyset$ in $\mathcal{G}_{0}.$
  \item For all $1<i<j<k\leqslant\ell-1,$ $N_{\mathcal{G}_{0}}(u_{i})\cap N_{\mathcal{G}_{0}}(u_{j})$ and $N_{\mathcal{G}_{0}}(u_{k})\cap N_{\mathcal{G}_{0}}(u)$ do not share a vertex in $G$.
\end{itemize}
   Next we show that the cardinality of $|U_{0}\setminus U|$ cannot be large, that means we can always guarantee $|U|$ is large enough. First, since the edge $u_{i}u_{j}$ in $W_{A}$ is light, $W(u_{i},u_{j})<2\binom{t}{2}.$ On the other hand, since $\mathcal{G}_{0}$ is $K$-almost regular, the degree of every vertex in $B$ is at most $K\delta,$ which implies that the number of vertices $u\in A$ such that $N_{\mathcal{G}_{0}}(u_{i})\cap N_{\mathcal{G}_{0}}(u_{j})\cap N_{\mathcal{G}_{0}}(u_{k})\neq\emptyset$ in $\mathcal{G}_{0}$ is at most $\binom{\ell-1}{2}\cdot 2\binom{t}{2}\cdot K\delta.$

  Second, for any $i,j,k\leqslant\ell-1,$ consider those two elements $x,y$ in $G$ which are in $v_{i,j}=(x,y)\in N_{\mathcal{G}_{0}}(u_{i})\cap N_{\mathcal{G}_{0}}(u_{j}).$ Observe that, each of $x$ and $y$ is contained in at most one vertex of $N_{\mathcal{G}_{0}}(u_{k})$ by Lemma~\ref{lem:disjoint}. Hence, for any fixed $i,j,k\leqslant\ell-1,$ there are at most $4\binom{t}{2}$ bad vertices in $B$ that we will not pick as $v_{k,\ell}$ of $H_{t}.$ Next we delete all neighbors of such bad vertices in $A$ from $U_{0}.$ Note that we need to do this for all distinct $i,j,k\leqslant \ell-1,$ hence we will delete at most $\binom{\ell-1}{3}\cdot 4\binom{t}{2}\cdot K\delta$ vertices from $U_{0}.$ Therefore, $|U_{0}\setminus U|\leqslant \binom{\ell-1}{2}\cdot 2\binom{t}{2}\cdot K\delta+\binom{\ell-1}{3}\cdot 4\binom{t}{2}\cdot K\delta.$ Since $m$ is sufficiently large, we have
\begin{equation*}
  (\frac{\delta^{2}}{64t^{3}m})^{\ell-1}|A|\geqslant 2\binom{\ell-1}{3}\cdot 4\binom{t}{2}\cdot K\delta+2\binom{\ell-1}{2}\cdot2\binom{t}{2}\cdot K\delta.
\end{equation*}
which implies that
\begin{equation*}
  |U|\geqslant\frac{|U_{0}|}{2}\geqslant\frac{1}{2}(\frac{\delta^{2}}{64t^{3}m})^{\ell-1}|A|.
\end{equation*}
Moreover, note that $\delta\geqslant c_{1}m^{\frac{t-2}{2t-3}}$ for some sufficiently large constant $c_{1}.$ Hence for any $2\leqslant\ell\leqslant t-1,$ we have
\begin{equation*}
  \frac{1}{2}(\frac{\delta^{2}}{64t^{3}m})^{\ell-1}|A|\geqslant\frac{8tm}{\delta}.
\end{equation*}
By Lemma~\ref{lem:light} and pigeonhole principle, there exists some $u_{\ell}\in U$ such that there are at least $\frac{\delta^{2}}{16t^{3}m}(|U|-1)\geqslant (\frac{\delta^{2}}{64t^{3}m})^{\ell}|A|$ light edges adjacent to $u_{\ell}$ in $U$. Now we have chosen the suitable $u_{\ell}$ with $2\leqslant\ell\leqslant t-1.$ This completes the recursive construction of $u_{1},u_{2},\ldots,u_{t-1}.$

Now we set $\ell=t-1$ and then there is a set $V\subseteq A$ with $|V|\geqslant (\frac{\delta^{2}}{64t^{3}m})^{t-1}\cdot |A|,$ such that for every $i\leqslant t-1$ and $v\in V,$ $u_{i}v$ is a light edge. Finally, we need to prove there is a vertex $u_{t}\in V$ such that for any $i<j<t,$ $N_{\mathcal{G}_{0}}(u_{i})\cap N_{\mathcal{G}_{0}}(u_{j})\cap N_{\mathcal{G}_{0}}(u_{t})=\emptyset$ in $\mathcal{G}_{0}$ and for any distinct $i,j,k<t$, $N_{\mathcal{G}_{0}}(u_{i})\cap N_{\mathcal{G}_{0}}(u_{j})$ and $N_{\mathcal{G}_{0}}(u_{k})\cap N_{\mathcal{G}_{0}}(u_{t})$ do not share a vertex in $G$. Using the similar argument and deletion method as above, we will delete at most $\binom{t-1}{3}\cdot 4\binom{t}{2}\cdot K\delta$ vertices from $V$. It is easy to see that such $u_{t}$ exists because $|V|\geqslant (\frac{\delta^{2}}{64t^{3}m})^{t-1}\cdot |A|> \binom{t-1}{3}\cdot 4\binom{t}{2}\cdot K\delta+\binom{t-1}{2}\cdot 2\binom{t}{2}\cdot K\delta.$ Hence there exists a copy of $H_{t}$ in which the vertices are pairwise disjoint, the proof is finished.
\end{proof}

\section{Conclusions and some open problems}\label{sec:conclusions}
Regarding the question about the function $f_{k}(n,H)$, there have been several interesting results and methods shown in \cite{Conlon2020, XuZhangJingGe2020, Janzer2020}. In this paper, we mainly focus on the case of $H=H_{t}$ is the $1$-subdivision of the complete graph $K_{t}$ and we prove that $f_{2}(n,H_{t})=\Omega(n^{1+\frac{1}{2t-3}}).$ Note that Theorem~\ref{thm:SomeknownResults} (i) gives that $f_{2}(n,H_{t})=O(n^{1+\frac{2}{t}}).$ Hence it will be interesting to determine the exponent $\beta$, if exists, such that $f_{2}(n,H_{t})=\Theta(n^{\beta}).$

 Theorem~\ref{thm:SomeknownResults} (i) also indicates that if $H$ is a bipartite graph with $e(H)\geqslant \frac{k}{k-1}|H|-\frac{2}{k},$ then $f_{k}(n,H)=\Theta(n).$ Our main result shows that, when $k=2,$ this bound can be pushed to $2|H|-2t,$ which answers a question of Conlon and Tyomkyn. It will be interesting to further decide whether this bound can be pushed closer to $\frac{k}{k-1}|H|,$ with $k\geqslant 3.$

In the classical Tur\'{a}n problem, there is a famous conjecture called rational exponent conjecture~\cite[Conjecture 1.6]{Erdos2013}, which states that for every rational number $r\in (1,2),$ there exists a single bipartite graph $H$ such that $\textup{ex}(n,H)=\Theta(n^{r}).$ This conjecture is still open, and the current progress of this conjecture can be seen in~\cite{ConlonJanzerLee2019, JMJ2020, JML2018, JQ19, KKL18} and the references therein. The known results show that something broadly similar holds for $f_{2}(n,H).$ Hence we think the following conjecture may be of interest.
\begin{conjecture}
  For every rational number $r\in (1,2),$ there exists a single bipartite graph $H$ such that
  \begin{equation*}
    f_{2}(n,H)=\Theta(n^{r}).
  \end{equation*}
\end{conjecture}

\section*{Acknowledgements}
Zixiang Xu is grateful to Oliver Janzer for his kind suggestions, and he would like to thank Yifan Jing for helpful discussions.

\bibliographystyle{abbrv}
\bibliography{Subdivisions}

\begin{thebibliography}{10}

\bibitem{Bukh2015}
B.~Bukh.
\newblock Random algebraic construction of extremal graphs.
\newblock {\em Bull. Lond. Math. Soc.}, 47(6):939--945, 2015.

\bibitem{ConlonJanzerLee2019}
D.~Conlon, O.~Janzer, and J.~Lee.
\newblock More on the extremal number of subdivisions.
\newblock {\em Combinatorica}, to appear (arXiv: 1903.10631, 2019).

\bibitem{ConlonLeeIMRN2018}
D.~Conlon and J.~Lee.
\newblock On the extremal number of subdivisions.
\newblock {\em Int. Math. Res. Not. IMRN}, to appear (arXiv: 1807.05008, 2018).

\bibitem{Conlon2020}
D.~Conlon and M.~Tyomkyn.
\newblock Repeated patterns in proper colourings.
\newblock {\em arXiv preprint}, arXiv: 2002.00921, 2020.

\bibitem{Erdos2013}
Z.~F\"{u}redi and M.~Simonovits.
\newblock The history of degenerate (bipartite) extremal graph problems.
\newblock In {\em Erd\"{o}s centennial}, volume~25 of {\em Bolyai Soc. Math.
  Stud.}, pages 169--264. J\'{a}nos Bolyai Math. Soc., Budapest, 2013.

\bibitem{XuZhangJingGe2020}
G.~Ge, Y.~Jing, Z.~Xu, and T.~Zhang.
\newblock Color isomorphic even cycles and a related {R}amsey problem.
\newblock {\em SIAM J. Discrete Math.}, to appear (arXiv: 2004.01932, 2020).

\bibitem{JanzerEJC2019}
O.~Janzer.
\newblock Improved bounds for the extremal number of subdivisions.
\newblock {\em Electron. J. Combin.}, 26(3):Paper 3.3, 6, 2019.

\bibitem{Janzer2020}
O.~Janzer.
\newblock Rainbow {T}ur\'{a}n number of even cycles, repeated patterns and
  blow-ups of cycles.
\newblock {\em arXiv preprint}, arXiv: 2006.01062, 2020.

\bibitem{JMJ2020}
T.~Jiang, Z.~Jiang, and J.~Ma.
\newblock Negligible obstructions and {T}ur\'{a}n exponents.
\newblock {\em arXiv preprint}, arXiv: 2007.02975, 2020.

\bibitem{JML2018}
T.~Jiang, J.~Ma, and L.~Yepremyan.
\newblock On {T}ur\'{a}n exponents of bipartite graphs.
\newblock {\em arXiv preprint}, arXiv: 1806.02838, 2018.

\bibitem{JQ19}
T.~Jiang and Y.~Qiu.
\newblock Many {T}ur\'{a}n exponents via subdivisions.
\newblock {\em arXiv preprint}, arXiv: 1908.02385, 2019.

\bibitem{Jiang2012}
T.~Jiang and R.~Seiver.
\newblock Tur\'{a}n numbers of subdivided graphs.
\newblock {\em SIAM J. Discrete Math.}, 26(3):1238--1255, 2012.

\bibitem{KKL18}
D.~Y. Kang, J.~Kim, and H.~Liu.
\newblock On the rational {T}ur\'{a}n exponent conjecture.
\newblock {\em arXiv preprint}, arXiv: 1811.06916, 2018.

\bibitem{LangWeil1954}
S.~Lang and A.~Weil.
\newblock Number of points of varieties in finite fields.
\newblock {\em Amer. J. Math.}, 76:819--827, 1954.

\bibitem{SudakovTomon2019}
B.~Sudakov and I.~Tomon.
\newblock Tur\'{a}n number of bipartite graphs with no {$K_{t,t}$}.
\newblock {\em Proc. Amer. Math. Soc.}, 148(7):2811--2818, 2020.

\end{thebibliography}

\end{document}